\documentclass[twoside]{amsart}
\usepackage{graphicx}
\usepackage{amssymb,amsfonts,amstext,amsmath,amsthm}
\newcommand{\ben}{\begin{enumerate}}
\newcommand{\een}{\end{enumerate}}
\input{amssym.def}
\input{amssym.tex}
\newtheorem{thm}{Theorem}[section]

\newtheorem{cor}[thm]{Corollary}
\begin{document}

\title{on totally antimagic total labeling of complete bipartite graphs }
\author{A. D. Akwu, D. O. A. Ajayi}
\date{}
\address{A. D. Akwu \hfill\break\indent Department of Mathematics, Federal University of Agriculture, Makurdi, Nigeria}
\address{Deborah Olayide A. AJAYI \hfill\break\indent Department
of Mathematics, University of Ibadan, Ibadan, Nigeria}
\email{ abolaopeyemi@yahoo.co.uk, adelaideajayi@yahoo.com}
\begin{abstract} This paper deals with the problem of finding totally antimagic total labelings of complete bipartite graphs. We prove that complete bipartite graphs are totally antimagic total graphs. We also show that the join of complete bipartite graphs with one vertex is a totally antimagic total graph.
\end{abstract}
\maketitle
\section{Introduction}
All graphs in this paper are finite, undirected and simple. If $G$ is a graph, then $V(G)$ and $E(G)$ stand for the vertex set and edge set of $G$. A $(p,q)$ graph $G$ is a graph such that $|V(G)|=p$ and $|E(G)|=q$. Let $K_{n,n}$ denote the complete bipartite graph with $n$ vertices in each partite set and $K_{n,m}$ denote the complete bipartite graph with different number of vertices in each partite set. For graph theoretic terminology, we refer to Chartrand and Lesniak \cite{CL}.

A {\it labeling} of graph $G$ is any mapping that sends some set of graph elements to a set of non-negative integers. If the domain is the vertex set or the edge set, the labeling is called {\it vertex labeling} or {\it edge labeling} respectively. Furthermore, if the domain is $V(G)\cup E(G)$, then the labeling is called a {\it total labeling}. If the vertices are labeled with the smallest possible numbers, i.e., $f(V(G))=\{1,2,3,...,p\}$, then the total labeling is called {\it super}.

Let $f$ be a vertex labeling of a graph $G$. Then we define the {\it edge-weight} of $uv\in E(G)$ to be $w_f(uv)=f(u)+f(v)$. If $f$ is total labeling, then the {\it edge-weight} of $uv$ is $wt_f(uv)=f(u)+f(uv)+f(v)$. The associated {\it vertex-weight} of a vertex $v$, $v\in V(G)$ is defined by $wt_f(v)=f(v)+\sum_{u\in N(v)}f(uv)$, where $N(v)$ is the set of the neighbors of $v$.
A labeling $f$ is called {\it edge-antimagic total (vertex-antimagic total)}, for short EAT (VAT), if all edge-weight  (vertex-weight) are pairwise distinct. A graph that admits EAT (VAT) labeling is called {\it EAT (VAT)} graph.

The notion of an antimagic graph was introduced by Hartsfield and Ringel \cite{HR}. According to Hartsfield and Ringel, an antimagic labeling of a graph $G$ is an edge labeling where all the vertex weights are required to be pairwise distinct. For an edge labeling, a {\it vertex weight} is the sum of the labels of all edges incident with the vertex. Bodendiek and Walther \cite{BW,BW1} were the first to introduce the concept of $(a,d)$-vertex antimagic edge labeling, they called this labeling {\it $(a,d)$vertex-antimagic labeling}. Ba\v{c}a {\it  et al}. \cite{BE} introduced the concept of an $(a,d)$-VAT labeling motivated by results of Bodendiek and Walther \cite{BW,BW1}. They defined the concept of {\it $(a,d)$-VAT labeling} as if the vertex-weights form an arithmetic sequence starting from $a$ and having a common difference $d$, i.e., the set of all the vertex-weight is $\{a,a+d,...,a+(p-1)d\}$ for some integers $a>0$ and $d\geq 0$. The basic properties of an $(a,d)$-VAT labeling and its relationship to other types of magic-type and antimagic-type labeling are investigated in \cite{BE}.

In \cite{SE}, it is shown how to construct super $(a,d)$-VAT labeling for certain families of graphs including complete graphs, complete bipartite graphs, cycles, paths and generalized Pertersen graphs. An {\it $(a,d)$-EAT labeling} of a graph $G$ is a total labeling with the property that the edge-weights form an arithmetic sequence starting from $a$ and having a common difference $d$, where $a>0$ and $d\geq 0$ are two integers. Simanjuntak {\t et al} \cite{SE1} introduced the notion of $(a,d)$-EAT labeling. They gave constructions of $(a,d)$-EAT labelings for cycles and paths. For further results on graph labelings see \cite{G} and \cite{MW}.

A {\it totally antimagic total labeling} (TAT) is a labeling that is simultaneously vertex-antimagic total and edge-antimagic total. A graph that admits totally antimagic total labeling is a {\it totally antimagic total graph} (TAT graph). The notion of totally antimagic total graphs was introduced by Ba\v{c}a {\it et al}. \cite{BE1}. They proved that paths, cycles, stars, double-stars and wheels are totally antimagic total graphs.

In Section $2$, we prove that complete bipartite graphs $K_{n,n}$ and $K_{n,m}$, where $n\neq m$ are totally antimagic total graphs. In Section $3$, we show that the join of complete bipartite graphs with one vertex is a TAT graph.
\section{TAT labelings of complete bipartite graphs}
We begin with the following definitions.
\subsection{Definition \cite{BE1}}
A labeling $g$ is {\it ordered (sharp ordered)} if $wt_g(u)\leq wt_g(v)\ \ (wt_g(u)<wt_g(v))$ holds for every pair of vertices $u,v\in V(G)$ such that $g(u)<g(v)$. A graph that admits a (sharp) ordered labeling is called a {\it (sharp) ordered graph}.
\subsection{Definition}
Let $n$ be a positive integer, a labeling $g$ of a graph $G$ is a {\it weak ordered (sharp)} if the vertex set of $G$ can be  partitioned  into $n$ sets such that any of the sets admits a (sharp) ordered labeling. A graph that admits a weak (sharp) ordered labeling is called a weak (sharp) ordered graph.\\
\\
Next, we show that $K_{n,n}$ is a TAT graph.

\begin{thm} \label{TM2}
Let $n>3$ be a positive integer, then the graph $K_{n,n}$ is a weak ordered super TAT graph.
\end{thm}
\begin{proof}
We partition the vertex set of $K_{n,n}$ into two equal sets $X_n$ and $Y_n$. Denote the vertices of $X_n$ and $Y_n$ by the symbols $u_i$ and $v_j$ respectively, where $i,j=1,2,...,n$, i.e., $X_n=\{u_1,u_2,...,u_n\}$ and $Y_n=\{v_1,v_2,...,v_n\}$.\\
We define a labeling $f$ of $K_{n,n}$ in the following way:
$$f(u_i)=2i-1,\ \  \ \ i=1,2,...,n,$$
$$f(v_j)= 2j, \ \ \ \ \ j=1,2,...,n,$$
$$f(u_iv_j)=n(2+j)-(i-1),\ \ \ \ i,j=1,2,...,n.$$
It is easy to see that $f$ is a bijection from the union of the vertex set and edge set of $K_{n,n}$ to the set $\{1,2,...,n(2+n)\}$.\\
\\
For the edge-weight of the edge $u_iv_j,\ \ i,j=1,2,...,n$ under the labeling $f$, we have
$$wt_f(u_iv_j)=f(u_i)+f(u_iv_j)+f(v_j)$$
$$=2i-1+n(2+j)-(i-1)+2j$$
$$=n(2+j)+2j+i.$$
The weights of all edges $u_iv_j\in E(K_{n,n})$ are different under the labeling $f$ which implies that the labeling $f$ is EAT.\\
\\
We have to check also that the vertex-weight are different. For the vertex $u_i$, 
we have
$$wt_f(u_i)=f(u_i)+\sum^n_{j=1}f(u_iv_j), \ \ i=1,2,...,n$$
$$=2i-1+\sum^n_{j=1(}n(2+j)-(i-1))$$
$$=i(2-n)+\frac{n}{2}(n^2+5n+2)-1.$$
Thus $wt_f(u_{i+1})=wt_f(u_i)-n+2$ which implies that $wt_f(u_i)>wt_f(u_{i+1})$, for $i=1,2,...,n-1$ (since $n>3$). Therefore the weights of vertices $u_i$ are all distinct.\\
\\
For the vertex $v_j$  
 we get
$$wt_f(v_j)= f(v_j)+\sum^n_{i=1}f(u_iv_j), \ \ \ j=1,2,...,n$$
$$=2j+\sum^n_{i=1}(n(2+j)-(i-1))$$
$$=j(2+n^2)+\frac{n}{2}(1+3n).$$
\\
Thus $$wt_f(v_j)=j(2+n^2)+\frac{n}{2}(1+3n)$$
$$=wt_f(v_{j+1})-n^2-2.$$
Therefore $wt_f(v_j)<wt_f(v_{j+1})$, for $j=1,2,...,n-1$ and the weights of vertices $v_j,\ \ j=1,2,...,n$ are all distinct.\\
\\
Next, we have to show that $wt_f(u_i)\neq wt_f(v_j)$.
Assume that $$wt_f(u_i)= wt_f(v_j)$$
$$\Rightarrow\ \ i(2-n)+\frac{n}{2}(n^2+5n+2)-1=j(2+n^2)+\frac{n}{2}(1+3n)$$
$$(2.1) \ \ \ \ \ \ \Rightarrow \ \ 2i(n-2)+2j(2+n^2)=n(n^2+2n+1)-2.$$
Case $1$: Use induction on $j$.
Let $j=1$ in Equation (2.1), we have
$$i=\frac{n^3+n-6}{2(n-2)}\notin \mathbb{Z^+}.$$
Since $(2(n-2))\nmid (n^3+n-6)$, this implies that $i\notin \mathbb{Z^+}$. Therefore $wt_f(u_i)\neq wt_f(v_j)$ for  $j=1$.

Assume that $wt_f(u_i)\neq wt_f(v_j)$ for $j=k$, that is
$wt_f(u_i)= wt_f(v_j)$ implies $i\notin \mathbb{Z^+}$.
 From  Equation (2.1), we have $$(2.2)\ \ \ \ \ i=\frac{n^3+2n^2+n-2-2k(n^2+2)}{2(n-2)}.$$
 Also, $(2(n-2))\nmid (n^3+2n^2+n-2-2k(n^2+2))$, which implies that $i\notin \mathbb{Z^+}$.
 Let $j=k+1$ in Equation (2.2), then we have $$i=\frac{n^3+2n^2+n-2-2(k+1)(n^2+2)}{2(n-2)}$$
  $$=\frac{n^3+n-6-2k(n^2+2)}{2(n-2)}.$$  Also, $(2(n-2))\nmid (n^3+n-6-2k(n^2+2))$, this implies that $i\notin \mathbb{Z^+}$. Thus $wt_f(u_i)\neq wt_f(v_j)$ for $j=k+1$. Therefore, $wt_f(u_i)\neq wt_f(v_j)$ for all values of $j$.

Case 2:  Use induction on $i$.
Let $i=1$ in Equation (2.1), then we have $$j=\frac{n^3+2n^2-n+2}{2(n^2+2)}.$$
Also, $(2(n^2+2))\nmid (n^3+2n^2-n+2)$ which implies that $wt_f(u_i)\neq wt_f(v_j)$ for $i=1$.

Assume that $wt_f(u_i)\neq wt_f(v_j)$ for $i=k$, that is
$wt_f(u_i)= wt_f(v_j)$ implies $j\notin \mathbb{Z^+}$.
 From  Equation (2.1), we have $$(2.3)\ \ \ \ \ j=\frac{n^3+2n^2+n-2-2k(n-2)}{2(n^2+2)}\notin \mathbb{Z^+}.$$
 Also, $(2(n^2+2))\nmid (n^3+2n^2+n-2-2k(n^2+2))$, which implies that $j\notin \mathbb{Z^+}$.

 Let $i=k+1$ in Equation (2.3), then we have $$j=\frac{n^3+2n^2+n-2-2(k+1)(n-2)}{2(n^2+2)}$$
  $$=\frac{n^3-n+2-2k(n-2)}{2(n^2+2)}.$$  Also, $(2(n^2+2))\nmid (n^3-n+2-2k(n-2))$, this implies that $j\notin \mathbb{Z^+}$. Thus $wt_f(u_i)\neq wt_f(v_j)$ for $j=k+1$. Therefore, $wt_f(u_i)\neq wt_f(v_j)$ for all values of $j$.


In conclusion, this shows that $wt_f(u_i)$ and $wt_f(v_j)$ are distinct.
As $wt_f(u_i)$ and $wt_f(v_j)$  are all different, this means that the labeling $f$ is VAT.\\
Also, since $wt_f(v_j)<wt_f(v_{j+1})$ where $f(v_j)<f(v_{j+1})$ and $wt_f(u_i)\nless wt_f(u_{i+1})$, for $f(u_i)<f(u_i+1)$, where $i,j=i,2,...,n-1$, this means that the graph $K_{n,n}$ is a weak ordered graph. \\
\\
Hence, since the labeling  $f$ is simultaneously EAT and VAT, the graph $K_{n,n}$ is a weak ordered super TAT graph.

\end{proof}

Next, we show that complete bipartite graph with different number of vertices in each partite set is a TAT graph.

\begin{thm}
Let $n,m$ be positive integers such that $n,m\geq 3$ and $n\neq m$, then the graph $K_{n,m}$ is a weak ordered super TAT  graph.
\begin{proof}
Partition the graph $K_{n,m}$ into two sets $X_n$ and $Y_m$. Denote the vertices of $X_n$ and $Y_m$ by the symbols $u_i$ and $v_j$ respectively, where $i=1,2,...,n$ and $j=1,2,...,m$. The number of vertices $p$ in $K_{n,m}$ is $n+m$ and the number of edges $q$ is $nm$. The edge set of $K_{n,m}$ is $u_iv_j$, where $i=1,2,...,n$ and $j=1,2,...,m$. We split the proof into the following three cases.\\
\\
Case $1$: $n$ even and $m$ a positive integer.\\
Consider the labeling $f:V(K_{n,m})\cup E(K_{n,m})\rightarrow \{1,2,...,p+q\}$ defined as
 $$ f(u_i)=i,\ \ \ i=1,2,...,n,$$
 $$f(v_j)=n+j,\ \ \ j=1,2,...,m,$$
 $$f(u_iv_j)=nj+m+i,\ \ \ j=1,2,...,m, i=1,2,...,n.$$
 The labeling $f$ is a bijection from the vertex set to the set $\{1,2,...,n+m\}$ which shows that $K_{n,m}$ is super. Also, labeling $f$ is a bijection from the union of the vertex set and the edge set of $K_{n,m}$ to the set $\{1,2,...,n(1+m)+m\}$.\\
 \\
 The edge-weight of the edge $u_iv_j,\ \ i=1,2,...,n, \ \ \ j=1,2,...,m$ is
 $$ wt_f(u_iv_j)=f(u_i)+f(u_iv_j)+f(v_j)$$
 $$=i+nj+m+i+n+j$$
 $$=2i+j(1+n)+n+m$$
 The weights of all the edges are different under the labeling $f$ which means that the labeling $f$ is EAT.\\
 \\
 Moreover, for the vertex $u_i,i=1,2,...,n$, we obtain the vertex weights as follows:
 $$wt_f(u_i)=f(u_i)+\sum_{j=1}^mf(u_iv_j)$$
 $$=i+\sum_{j=1}^m(nj+m+i).$$
 $$(2.4)\ \ \ \ \ \ \ \ =i(1+m)+\frac{m}{2}(m(n+2)+n)$$
 For the vertex $v_j, j=1,2,...,m$, we obtain the vertex-weight
 $$wt_f(v_j)=f(v_j)+\sum_{i=1}^nf(u_iv_j),$$
 $$=n+j+\sum_{i=1}^n(nj+m+i)$$
 $$(2.5)\ \ \ \ \ \ =j(1+n^2)+\frac{n}{2}(2m+n+3).$$
 Next, we show that the vertex weights $wt_f(u_i)$ and $wt_f(v_j)$ are distinct.

 Assume that $wt_f(u_i)=wt_f(v_j)$, then we have
 $$i(1+m)+\frac{m}{2}(m(n+2)+n)=j(1+n^2)+\frac{n}{2}(2m+n+3)$$
 $$(2.6) \ \ \ \ \ \ \Rightarrow 2i(1+m)-2j(1+n^2)=nm-m^2(n+2)+n(n+3)$$
 Case 1: Use induction on $j$ in Equation (2.6). Let $j=1$ in Equation (2.6), to give
 $$i=\frac{-m^2(n+2)+nm+3n(1+n)+2}{2(1+m)}.$$
 Since $(2(1+m))\nmid (-m^2(n+2)+nm+3n(1+n)+2)$, this implies that $i\notin \mathbb {Z^+}$. Thus $wt_f(u_i)\neq wt_f(v_j)$ for $j=1$.

 Assume that $wt_f(u_i)\neq wt_f(v_j)$ for $j=k$, that is
$wt_f(u_i)= wt_f(v_j)$ implies $i\notin \mathbb{Z^+}$.
 From  Equation (2.6), we have $$(2.7)\ \ \ \ \ i=\frac{-m^2(n+2)+nm+n(3+n)+2k(1+n^2)}{2(1+m)}\notin \mathbb{Z^+}.$$
 Also, $(2(1+m))\nmid (-m^2(n+2)+nm+n(3+n)+2k(1+n^2))$, which implies that $i\notin \mathbb{Z^+}$.

Let $j=k+1$ in Equation (2.7), then we have
 $$i=\frac{-m^2(n+2)+nm+3n(1+n)+2k(1+n^2)+2}{2(1+m)}\notin \mathbb {Z^+}.$$
 Therefore, $wt_f(u_i)\neq wt_f(v_j)$ for $j=k+1$. Thus $wt_f(u_i)\neq wt_f(v_j)$ for all values of $j$.

 Case 2: Use induction on $i$ in Equation (2.6). Let $i=1$ in Equation (2.6), to give
 $$j=\frac{-n^2-n(m-m^2+3)+2m(1+m)+2}{2(1+n^2)} \notin \mathbb {Z^+}.$$
 Therefore, $wt_f(u_i)\neq wt_f(v_j)$ for $i=1$.

 Assume that $wt_f(u_i)\neq wt_f(v_j)$ for $i=k$, that is
$wt_f(u_i)= wt_f(v_j)$ implies $j\notin \mathbb{Z^+}$.
 From  Equation (2.6), we have $$(2.8)\ \ \ \ \ j=\frac{-n^2-n(m-m^2+3)+2m^2+2k(1+m)}{2(1+n^2)}\notin \mathbb{Z^+}.$$
 Also, $(2(1+n^2))\nmid (-n^2-n(m-m^2+3)+2m^2+2k(1+m))$, which implies that $j\notin \mathbb{Z^+}$.
Let $i=k+1$ in Equation (2.8), then we have
 $$j=\frac{-n^2-n(m-m^2+3)+2((1+m)(m+k)+1)}{2(1+n^2)} \notin \mathbb {Z^+}.$$
 This implies that $wt_f(u_i)\neq wt_f(v_j)$ for $i=k+1$. Thus $wt_f(u_i)\neq wt_f(v_j)$ for all values of $i$.
 Therefore the weights of the vertices are distinct under the labeling $f$ which implies that the labeling $f$ is VAT.\\
 \\
 Also, from Equation (2.4), $wt_f(u_i)<wt_f(u_{i+1})$ whenever $f(u_i)<f(u_{i+1})$ and from Equation (2.5), $wt_f(v_j)<wt_f(v_{j+1})$ whenever $f(v_j)<f(v_{j+1})$. This implies that $K_{n,m}$ is a weak ordered graph.\\
 \\
 Finally, since the graph $K_{n,m}$ admits the labeling $f$ that is simultaneously EAT and VAT, then $K_{n,m}$ is a TAT graph.\\
\\
 Case $2$: $n$ odd and $m$ even.\\
 By interchanging the partition set $X_n$ with $Y_m$ and $Y_m$ with $X_n$ in case $1$ above gives the result.\\
\\
 Case $3$: $n,m$ odd. Without loss of generality, $m>n$.\\
 Consider the labeling $f:K_{n,m}\rightarrow \{1,2,...,p+q\}$ defined as follows:
 $$f(u_i)=n+m-2i+1,\ \ \ i=1,2,...,n,\ \ \ $$
 $$f(v_j)=\left\{\begin{array}{c}{\hspace{-2cm}}j,\ \ \ \ \ \  j=1,2,...,m-n\\ n+2j-m,\ \ \ m-n+j=1,2,...,m \end{array}\right.$$
$$f(u_iv_j)=m+nj+i,\ \ \ j=1,2,...,m,\ \ \ i=1,2,...,n.$$
It is easy to see that the labeling $f$ above is super.\\
\\
The edge-weight of the edge $u_iv_j, \ \ i=1,2,...,n,\ \ j=1,2,...,m$ is given as:
$$wt_f(u_iv_j)=f(u_i)+f(u_iv_j)+f(v_j)$$
$$=\left\{\begin{array}{c}{\hspace{-0.5cm}}n(1+j)+2m-1+j+1,\ \ \ i=1,2,...,n,\ \ \ j=1,2,...,m-n,\\n(2+j)+m-
i+2j+1,\ \ \ i=1,2,...,n,\ \ \ m-n+j=1,2,...,m\end{array}\right.$$
The weights of all the edges in $K_{n,m}$ are different under the labeling $f$ which means that the labeling $f$ is EAT.\\
\\
We have to check also that the vertex-weights are different. \\
For vertex $u_i,\ \ \ i=1,2,...,n$, we get:
$$wt_f(u_i)=f(u_i)+\sum^m_{j=1}f(u_iv_j)$$
$$=n+m-2i+1+\sum_{j=1}^m(m+nj+i)$$
$$(2.9)\ \ \ \ \ \ \ \ =i(m-2)+\frac{m}{2}((1+m)(n+2))+n+1$$
Also, for vertex $v_j, \ \ \ j=1,2,...,m$, we get
$$wt_f(v_j)=f(v_j)+\sum^n_{i=1}f(u_iv_j)$$
$$(2.10)\ \ \ \ \ \ \ \ =\left\{\begin{array}{c}{\hspace{-0.5cm}}j(1+n^2)+\frac{n}{2}(1+n+2m),\ \ \ j=1,2,...,m-n,\\j(2+n^2)+\frac{n}{2}(3+n+2m)-m,\ \ \ m-n+j=1,2,...,m\end{array}\right.$$
From Equation (2.9), we have that $wt_f(u_i)<wt_f(u_{i+1}),\ \ \ i=1,2,...,n-1$, whenever $f(u_i)<f(u_{i+1})$. Also, from Equation (2.10), we have that $wt_f(v_j)<wt_f(v_{j+1}), \ \ \ j=1,2,...,m-1$, whenever $f(v_j)<f(v_{j+1})$. Therefore, $K_{n,m}$ is a weak ordered graph.
Hence, $K_{n,m}$ is a TAT graph since the labeling $f$ is both EAT and VAT.
\end{proof}
\end{thm}

\section{TAT labeling of join of graphs}
Let $G\cup H$ denote the disjoint union of graphs $G$ and $H$. The join $G\oplus H$ of the disjoint graphs $G$ and $H$ is the graph $G\cup H$ together with all the edges joining vertices of $V(G)$ and vertices of $V(H)$. In this section, we deal with a totally antimagic total labeling of $K_{n,n}\oplus K_1$ and $K_{n,m}\oplus K_1$. According to Miller {\it et al} \cite{ME1}, we have that all graphs are (super) EAT. If there exist a super EAT labeling of a graph $G$ satisfying the additional condition that it is a weak ordered, we are able to prove that the join $K_{n,n}\oplus K_1$ and $K_{n,m}\oplus K_1$ are TAT. \\

\begin{thm}
Let $K_{n,n}, n>3$ be a weak ordered super EAT graph. Then $K_{n,n}\oplus K_1$ is a TAT graph.
\end{thm}
\begin{proof}
Let $f$ be a weak ordered super EAT labeling of $K_{n,n}$. As $f$ is super, we can denote the vertices of $K_{n,n}$by the symbols $\{u_1,u_2,...,u_n\}$ and $\{v_1,v_2,...,v_n\}$ such that
$$f(u_i)=2i-1,\ \ \ \ \ i=1,2,...,n$$
$$f(v_j)=2j,\ \ \ \ \ j=1,2,...,n$$
Since $f$ is a weak ordered, then for $j=1,2,...,n-1$, we have that $wt_f(v_j)\leq wt_f(v_{j+1})$ which follows from Theorem \ref{TM2}.\\
\\
By the symbol $u$, we denote the vertex of $K_{n,n}\oplus K_1$ not belonging to $K_{n,n}$.\\
\\
We define a new labeling $g$ of $K_{n,n}\oplus K_1$ such that
$$g(x)=f(x),\ \ \ \ \ \ x\in V(K_{n,n})\cup E(K_{n,n})$$
$$g(u)=n(n+2)+1$$
$$g(u_iu)=n^2+2(n+i),\ \ \ \ \ i=1,2,...,n$$
$$g(v_ju)=n^2+2(n+j)+1,\ \ \ \ j=1,2,...,n$$
It is easy to see that $g$ is a bijection from $V(K_{n,n}\oplus K_1)\cup E(K_{n,n}\oplus K_1)$ to the set $\{1,2,...,n(n+4)+1\}$.\\
\\ For the vertex-weights under labeling $g$, we have the following:
$$wt_g(u)=g(u)+\sum^n_{i=1}g(u_iu)+\sum^n_{j=1}g(v_ju)$$
$$=n(n+2)+1+\sum^n_{i=1}(n^2+2(n+i))+\sum^n_{j=1}(n^2+2(n+j)+1)$$
$$=(1+2n)(n^2+3n+1)$$
\\
For vertex $u_i, \ \ \ \i=1,2,...,n$, we have:
$$wt_g(u_i)=g(u_i)+\sum_{v_j\in N(u_i)}g(u_iv_j)+g(u_iu)$$
$$=wt_f(u_i)+2n+n^2+i+1$$
$$\geq wt_f(u_{i+1})+2n+n^2+i+1$$
$$>wt_f(u_{i+1})+2n+n^2+i+2$$
$$=wt_g(u_{i+1})$$
\\
Moreover, for $j=1,2,...,n$, we get
$$wt_g(v_j)=g(v_j)+\sum_{u_i\in N(v_j)}g(v_ju_i)+g(v_ju)$$
\\
We have $wt_f(v_j)$ from Theorem \ref{TM2}, hence we have
$$wt_g(v_j)=wt_f(v_j)+n^2+3n+j+1$$
$$\leq wt_f(v_{j+1})+n^2+3n+j+1$$
$$<wt_f(v_{j+1})+n^2+3n+j+2$$
$$=wt_g(v_{j+1})$$
\\
From Theorem (2.3), $wt_g(u_i)\neq wt_g(v_j)$. Also, $wt_g(u_i)>wt_g(u)$ and $wt_g(v_j)>wt_g(u)$.
Thus the vertex-weights are all different which implies that the labeling $g$ is VAT.\\
\\
The edge-weights of the edges in $E(K_{n,n})$ under the labeling $g$ are all different as $f$ is an EAT labeling of $K_{n,n}$. More precisely, we have
$$wt_g(e)=wt_f(e)$$ for every $e\in E(K_{n,n})$\\
\\
Also, as $f$ is super, for the upper bound on the maximum edge-weight of $e\in K_{n,n}$ under the labeling $g$, we have $$wt_g^{ max}(e)=wt_f^{max}(e)\leq n^2+6n-1 $$
\\
For $i=1,2,...,n$, we get
$$wt_g(uu_i)=g(u)+g(uu_i)+g(u_i)$$
\\
From Theorem \ref{TM2}, we have $f(u_i)=2i-1$. Since $f(x)=g(x)$ for $x\in V(K_{n,n})\cup E(K_{n,n})$, we have that $$g(u_i)=2i-1$$.
$$wt_g(uu_i)=(n^2+2n+1)+(n(n+2)+1)+2i-1$$
$$=2n^2+4(n+i)>n(n+6)-1\geq wt_g^{max}(e)$$
Where $e\in E(K_{n,n})$.\\
\\
Moreover, for $j=1,2,...,n$, we get
$$wt_g(uv_j)=g(u)+g(uv_j)+g(v_j)$$
From Theorem \ref{TM2}, we have $f(v_j)=2j$. Since $f(x)=g(x)$ for $x\in V(K_{n,n})\cup E(K_{n,n})$, we have that $g(v_j)=2j$.
$$wt_g(uv_j)=(n^2+2n+1)+(n^2+2(n+j)+1)+2j$$
$$=2(n^2+1)+4(n+j)>n(n+6)-1\neq wt_g(uu_i) \geq wt_g^{max}(e)$$
Where $e\in E(K_{n,n})$.\\
\\
Now to show that the edge-weights are all different.
\\
Assume that $wt_g(u_iu)=wt_g(v_ju)$, then we have
$$2n^2+4(n+i)=2(n^2+1)+4(n+).$$
 This gives $i-j=\frac{1}{2}\notin \mathbb{Z^+}$. Therefore, $wt_g(u_iu)\neq wt_g(v_ju)$ which implies that the labeling $g$ is an EAT labeling.
Thus $g$ is a TAT labeling of $K_{n,n}\oplus K_1$.

\end{proof}
\begin{cor}
Let $K_{n,m}$ be a weak ordered super EAT graph. Then $K_{n,m}\oplus K_1$ is a TAT graph.
\end{cor}

\end{document}